\documentclass{article}
\usepackage{amsmath,amssymb,color}
\usepackage{graphicx}
\usepackage{epsfig}
\usepackage{fullpage}
\usepackage{cite}
\usepackage{palatino}
\numberwithin{equation}{section}

\newcommand{\D}{\displaystyle}
\newcommand{\pd}[2]{\D\frac{\partial{#1}}{\partial{#2}}}  
\newcommand{\pdtwo}[3]{\D\frac{\partial^2{#1}}{\partial{#2}\partial{#3}}}

\newcommand{\qed}{\hfill$\blacksquare$}

\newcommand{\N}{\mathbb{N}}

\newcommand{\R}{\mathbb{R}}

\newcommand{\T}{\mathbb{T}}
\def\erf{\mathop{\rm erf}\nolimits}
\def\Tr{\mathop{\rm Tr}\nolimits}

\newtheorem{thm}{Theorem}[section]
\newtheorem{lem}[thm]{Lemma}
\newtheorem{prop}[thm]{Proposition}
\newtheorem{corr}[thm]{Corollary}
\newtheorem{define}{Definition}[section]

\newtheorem{defn}{Definition}[section]
\newtheorem{rem}{Remark}[section]
\newtheorem{remark}{Remark}[section]
\newenvironment{proof}{{\bf Proof.}  }{\hfill$\blacksquare$}

\renewcommand{\a}{\alpha}
\renewcommand{\b}{\beta}
\newcommand{\g}{\gamma}
\renewcommand{\d}{\delta}

\renewcommand{\k}{\kappa}

\renewcommand{\th}{\theta}

\newcommand{\1}{\mathbf{1}}
\newcommand{\0}{\mathbf{0}}
\newcommand{\av}[1]{\left|{#1}\right|}

\newcommand{\ip}[2]{\left\langle{{#1}},{{#2}}\right\rangle}
\newcommand{\norm}[1]{\left\|{#1}\right\|}
\renewcommand{\l}{\left}
\newcommand{\q}{\quad}
\renewcommand{\r}{\right}

\newcommand{\be}{\begin{enumerate}}
\newcommand{\bi}{\begin{itemize}}
\newcommand{\ee}{\end{enumerate}}
\newcommand{\ei}{\end{itemize}}

\newcommand{\x}{{\mathbf {{x}}}}

\renewcommand{\P}{\mathbb{P}}

\renewcommand{\phi}{\varphi}
\newcommand{\bomega}{{\boldsymbol{\omega}}}
\newcommand{\btheta}{{\boldsymbol{\theta}}}
\newcommand{\bchi}{{\boldsymbol{\chi}}}

\newcommand{\f}{{\mathbf{f}}}

\newcommand{\ST}{\mathcal{S}_\btheta}
\newcommand{\SO}{\mathcal{S}_\bomega}
\newcommand{\psync}{\mathcal{P}_{{{\mathsf{sync}}}}}
\newcommand{\M}{\mathbf{M}}
\newcommand{\DD}{\mathbf{D}}
\newcommand{\gstar}{\gamma^\star}
\newcommand{\gmin}{{\gamma}_{{{\mathsf{min}}}}}
\newcommand{\gmax}{{\gamma}_{{{\mathsf{max}}}}}

\newcommand{\omin}{{\bomega}_{{{\mathsf{min}}}}}
\newcommand{\omax}{{\bomega}_{{{\mathsf{max}}}}}
\renewcommand{\v}{{\mathbf{v}}} \newcommand{\w}{{\mathbf{w}}} 

\begin{document}

\title{Fully synchronous solutions and the synchronization phase
  transition for the finite-$N$ Kuramoto model} \date{\today}

\author{Jared C. Bronski \\ University of Illinois \and Lee DeVille \\
  University of Illinois \and Moon Jip Park \\ University of Illinois}

\maketitle

\begin{abstract}
  We present a detailed analysis of the stability of synchronized
  solutions to the Kuramoto system of oscillators. We derive an
  analytical expression counting the dimension of the unstable manifold  
associated to a given stationary solution. From this we are able
  to derive a number of consequences, including: analytic expressions
  for the first and last frequency vectors to synchronize, upper and
  lower bounds on the probability that a randomly chosen frequency
  vector will synchronize, and very sharp results on the large $N$
  limit of this model. One of the surprises in this calculation is
  that for frequencies that are Gaussian distributed the correct scaling 
  for full synchrony is not the one commonly
  studied in the literature---rather, there is a logarithmic
  correction to the scaling which is related to the extremal value
  statistics of the random frequency vector. 
\end{abstract}

\section{Introduction}

\subsection{History of Kuramoto model}

The study of synchronization of coupled nonlinear oscillators has a
history that spans several centuries, starting with Huygens'
observation of synchronizing pendulum clocks~\cite{Huygens,
  Bennett.etal.02}.  There has been a great body of work throughout
this history studying such synchronization phenomena; for reviews
see~\cite{Sync.book, PRK.book, Winfree.book}.  There are a wide
variety of such models derived from a diverse collection of
mathematical and scientific questions, including
pulse-coupled~\cite{Knight.72, Peskin.75} and conservative~\cite{FPU,
  Ford.92} models.  These models arise in a number of contexts, including numerous biological 
applications~\cite{Kopell.Ermentrout.86,Medvedev.Kopell.01}, and exhibit a diversity of behaviors.

A fundamental model of {\em synchronization}, however, is the case
where we consider independent oscillators connected through
dissipative coupling, and the fundamental question is~\cite{Sync.book,S}:
how do independent oscillators with different frequencies adjust
themselves to produce a collective mode?

The system of coupled ordinary differential equations
\[
\frac{d\theta_i}{dt} = \gamma \sum_j \sin(\theta_j-\theta_i) \] was
first proposed by Kuramoto~\cite{Kuramoto.book, Kuramoto.91} as a
model for the synchronization of oscillators; this model, and
variants, are widely known as the Kuramoto model. Since this time, the
Kuramoto model has been a fundamental model for many types of systems
exhibiting synchronization~\cite{BS, E, TOR, Hansel.Sompolinsky.92}
and related phenomena such as flocking~\cite{HK, HLRS}.  For reviews,
see~\cite{S, Acebron.etal.05}.  Most work in this area has analyzed
the Kuramoto model in the continuum limit, where the number is
oscillators $N$ is formally allowed to go to infinity and the
summation is replaced by an appropriate integral. Such formal
calculations provide a tremendous amount of physical insight into the
problem; however, it has proven to be difficult to make these
approaches rigorous.

An alternative approach, attributed by Strogatz~\cite{S} to a series of 
lectures by Kopell, is to analyze the
finite-$N$ problem carefully and then take the $N\to\infty$
limit. There have been a few papers which have established rigorous
results on the existence and stability of synchronized
solutions~\cite{VO1,VO2, Mirollo.Strogatz.05,Mirollo.Strogatz.90} but to date this the finite-$N$
problem has been only partially understood. In this paper we carry out
a substantial portion of this program. We focus
on the case of full synchrony, where all of the oscillators rotate
with the same angular frequency. We derive new characterizations of
the stable regions, which permit a detailed understanding of the
region in frequency space where full synchronization occurs.  In
particular we are able to identify a large subset of the stable
region, in the form of the Voronoi cell of a well-understood high
dimensional lattice. Using this we find upper and lower bounds on the
probability that the system undergoes full synchronization.

\subsection{ Problem Formulation}

The standard formation of the Kuramoto problem is as follows: Consider
a weighted (directed) graph $G=(V,E)$ and denote $\gamma_{ij}>0$ as
the weight of edge $i\to j$.  For any $\omega\in\R^N$, define the
dynamical system on $\theta\in\T^N$ by
\begin{equation}\label{eq:K-gen}
  \frac{d}{dt}\theta_i = \omega_i + \sum_{j} \gamma_{ji} \sin(\theta_j-\theta_i).
\end{equation}
One of the more common choices of interaction graph is the symmetric
all-to-all graph where we assume that all the $\gamma_{ij}$ are equal.
We thus consider the system
\begin{equation}\label{eq:K}
 \frac{d}{dt} \theta_i = \omega_i + \gamma\sum_j \sin(\theta_j-\theta_i),
\end{equation}
or, defining the function $\f\colon \T^n\to \R^n$ coordinatewise as
\begin{equation}\label{eq:defofF}
  f_i(\btheta) = \sum_j \sin(\theta_j-\theta_i), 
\end{equation}
we can write~\eqref{eq:K} as 
\begin{equation}\label{eq:Kvec}
  \frac{d}{dt} \btheta = \bomega +\gamma \f(\btheta).
\end{equation}

\begin{remark}
  Note that we have not rescaled the coupling coefficient
  in~\eqref{eq:K}; the typical scaling chosen for Kuramoto is
  $\gamma/N$.  We will refer to this scaling below as the ``classical
  scaling''.  One of the results of this paper is that the classical
  scaling is only the correct scaling for certain problems, and not
  for others. For the problem where the frequencies $\omega_i$ are
  chosen from a Gaussian distribution, for instance, we shall see that
  the correct scaling differs from the classical one by a logarithmic
  term.
\end{remark}

\begin{remark}
 It is worth noting that this flow is a gradient flow, if one considers the angles as lying in the covering space 
${\mathbb R}^N$ rather than ${\mathbb T}^N$. The flow can be written as 
\[
 \frac{d\btheta }{dt} = \nabla L
\]
where the energy $L$ is given by 
\[
 L = \langle \bomega, \btheta\rangle + \gamma \sum_{i,j} \cos(\theta_i-\theta_j)-N. 
\]
Here the constant $-N$ is chosen for convenience and obviously doesn't influence the dynamics. 
It is common to work with the ``order parameter'' $R$ 
\[
 R^2(\btheta) = \left(\sum_i \cos(\theta_i) \right)^2 +  \left(\sum_i \sin(\theta_i) \right)^2 = N + \sum_{i,j} \cos(\theta_i-\theta_j),
\]
which gives the energy function as 
\[
 L = \langle \bomega, \btheta\rangle + R^2(\btheta). 
\]
Thus the Kuramoto flow tries to maximize an energy given by a sum of two terms. The first term acts to 
align the flow with the frequency vector $\bomega,$ while the second acts to increase the order parameter. 
All of the structure that arises in Kuramoto model is due to a competition between these two effects. 
 
\end{remark}

The fundamental question we consider is the question of
whether~\eqref{eq:K} admits a fully synchronous solution, as we now
define:
\begin{define}
  For a given value of the frequency vector $\bomega$ we say that the
  Kuramoto model exhibits {\em full synchrony} if~\eqref{eq:K} with
  $\gamma=1$ admits a stationary solution $\btheta$ which is
  asymptotically stable, i.e. we have a solution to the equation
  \begin{equation}\label{eq:fp}
    \f(\btheta) = -\bomega
  \end{equation}
  such that, if we define the Jacobian matrix
  \begin{equation}\label{eq:defofJ}
    J = \nabla_\btheta f,
  \end{equation}
  then $J$ is negative semi-definite with a one-dimensional kernel.
  If such a solution exists, we call it a {\em fully synchronous
    solution}.  
\end{define}

\begin{rem}
  It is not hard to see that, for fully synchronous solutions, we will
  have $$\max_{i,j}\sup_{t\in[0,\infty)}\norm{\theta_i(t)-\theta_j(t)}
  < \infty.$$ However, note that due to the continuous symmetry
  of~\eqref{eq:K}, these fully synchronous solutions can be moving in
  the same co-rotating frame.  The Kuramoto model also permits
  partially synchronous solutions where a large mass of the
  oscillators are fixed relative to each other, and some subset
  precesses relative to them.  We do not consider such solutions in
  this work, q.v.~the discussion in Section~\ref{sec:scaling} below.
\end{rem}

The results of the paper are organized as follows.  In
Section~\ref{sec:stable} we give an almost complete characterization
of the $\bomega$ which give rise to fully synchronous solutions; in
particular, we show that this set is convex, identify many important
points on the boundary of this set, and explicitly determine the
convex hull of these points.  We would like to characterize ``how
easy'' it is for a Kuramoto problem to fully synchronize, but there
are many senses in which one can pose this problem; we consider two
different cases below.

For example, we could pose the following question: given
$\bomega\in\R^N$, define $\gstar(\bomega)$ as the minimal $\g$ for
which~\eqref{eq:Kvec} has a fully synchronous solution.  Of course the
presence of synchrony is invariant under a rescaling in time, so one
may as well assume that $\av{\bomega} = 1$. The motivates the
definition
\begin{define}
  Define the lower and upper critical couplings $\gmin(N)$ and
  $\gmax(N)$ in the following way:
\begin{equation}\label{eq:defofgmingmax}
  \gmin(N) := \inf_{\substack{\bomega\in\R^N\\\av{\bomega} = 1}} \gstar(\bomega), \q \gmax(N) := \sup_{\substack{\bomega\in\R^N\\\av{\bomega} = 1}} \gstar(\bomega)
\end{equation}
\end{define}
Roughly $\gmin(N)$ characterizes the minimal coupling constant we
should choose to get full synchrony, and the values of $\bomega$ for
which the infimum is attained (by symmetry there are many) can be
thought of as the ``most easily synchronizable frequency(s)'';
analogously $\gmax(N)$ characterizes the most difficult
frequencies. to synchronize, and gives the largest necessary coupling
constant. For coupling constants above this value all frequency
vectors synchronize.  The quantities $\gmin$ and $\gmax$ are closely
related to the quantity $K_c$ studied by Verwoerd and Mason in their
work\cite{VO1,VO2}: $\gmin(N)$ is essentially the minimum of $K_c$
over the unit circle, and $\gmax$ the maximum.

We study $\gmin(N),\gmax(N)$ in Section~\ref{sec:upperlower}, and give
bounds which show the different asymptotic scalings of these
quantities.  We will show that in the limit $N\to\infty$,
$\gmin(N)\sim N^{-3/2}$ and $\gmax(N)\sim N^{-1}$.  In particular,
synchrony ``turns on'' much earlier than the classical scaling,
although it ``fills up'' in a way consistent with the classical
scaling.

A different method of choosing $\bomega$, which is very common in the
literature, is to assume that the entries of $\bomega$ are chosen
independently from some probability distribution, i.e. that the
$\omega_n$ are iid random variables.  In Section~\ref{sec:Gaussian} we
consider this question; stated precisely, we proceed as follows.  Fix
$\g>0$ and $N$, and choose $\omega_n$ independently from a Gaussian
distribution with mean zero and unit variance, i.e. the $\omega_n$ are
independent $N(0,1)$.  Define $\psync(\g,N)$ as the probability
of~\eqref{eq:K} having a fully synchronous solution. Define
$\varphi(N) := \sqrt{2\ln(N)}/(N+1)$, and what we show below in
Theorem~\ref{thm:limit} is
\begin{equation}\label{eq:limit}
  \lim_{N\to\infty} \psync(\delta \varphi(N), N) = \begin{cases} 0,& \delta < 1,\\ 1,& \delta > 2.\end{cases}
\end{equation}
In particular, this implies that $\psync(\gamma/N,N) \to 0$ for any
$\gamma$, so that in the classical scaling the probability of full
synchrony is zero.  In fact, we will show in
Proposition~\ref{prop:exp} that the probability of full synchrony
decays to zero exponentially fast as $N\to\infty$.  We show that this
anomalous scaling is closely related to the extreme value statistics
for a Gaussian distribution, and is to be expected for any
distribution of frequencies which is not of compact support.

Finally, in Section~\ref{sec:scaling} we finish with some comments
about the different choices of scaling and why they arise and present
a few open questions.

\section{Characterization of the Stable Set}\label{sec:stable}

In this section we write down a relatively complete description of the
set of $\bomega$ for which~\eqref{eq:K} admits a stable fully synchronous
solution. We do this by proving an index theorem which counts the dimension 
of the unstable manifold to any stationary solution. The stable synchronous 
solutions are then those for which the unstable manifold is zero dimensional.

\subsection{Notation}

We consider the finite $N$ Kuramoto model with uniform sine coupling
on the complete graph, which we repeat
\begin{equation}\tag{\ref{eq:Kvec}}
  \frac{d}{dt} \btheta = \bomega +\gamma \f(\btheta)
\end{equation}
Note that 
\begin{equation*}
  \sum_{i=1}^N f_i(\btheta) = 0
\end{equation*}
for all $\btheta$, due to a telescoping sum.  Therefore we have 
\begin{equation}\label{eq:defofOmega}
  \frac{d}{dt} \sum \theta_i = \sum \omega_i =: \Omega.
\end{equation}
This means that, if $\Omega\neq0$, the center of mass of the system
precesses around the circle with a constant velocity.  Using the
change of variables $\tilde\theta_i = \theta_i - N^{-1}\Omega t$ puts
us into a corotating frame and allows us to assume without loss of
generality that $\sum_i \omega_i=0$.


The function $\f\colon \T^n\to\R^n$ is a natural map from the
configuration space ${\mathbb T}^{n}$ to the frequency space ${\mathbb
  R}^n$.  Assuming some convexity conditions which will be established
later, the frequency vector $\bomega$ is the convex dual variable to
the angle vector $\btheta$.  It will also be convenient to think of
this as a map $\f\colon \T^{n-1} \to \R^{n-1} \cong
\R^n/\{1,1,\dots,1\}$ --- this can be done, for example, by fixing a
single $\theta_n$.

The Jacobian of this mapping controls the stability of the
synchronized state, and it is crucial to have a good understanding of
set of parameter values for which the Jacobian is negative
semi-definite. This motivates the following definition:

\begin{define}
  We define ${\mathcal S}_\theta$ to be the set of configurations for
  which the Jacobian $\frac{\partial \f}{\partial {\bf \theta}}$ is
  negative semi-definite with a one dimensional kernel. The boundary
  of this set is obviously the set of configurations for which the
  Jacobian $\frac{\partial \f}{\partial {\bf \theta}}$ is negative
  semi-definite with a kernel of dimension two or more.
\end{define}

Throughout the paper calligraphic capital letters will denote sets of
particular interest, with a subscript of $\theta$ denoting sets in the
configuration space $\T^{n-1}$ and a subscript $\omega$ denoting the
corresponding sets in the frequency space. For instance if $\ST$
denotes the set of asymptotically stable stationary configurations
then $\SO$, the set of frequencies admitting an asymptotically stable
stationary configuration, is simply the image of $\ST$ under the map
$\f$:
\[
\f: \ST \mapsto \SO
\]    

It is clear that $\SO$ is the important object for studying
synchronization: all questions about the probability of (full)
synchrony are questions about the size of $\SO$ in some measure. One
of the key ingredients in this is a good characterization of $\SO$.

\subsection{Index Theorem}

In this section we prove an index theorem which counts the dimension
of the unstable manifold to the synchronized state. The basic idea of
this section is that the Jacobian takes a relatively simple form, as
it can be written as a rank two perturbation of a diagonal matrix.  An
application of a rank-two perturbation formula gives a straightforward
count of the number of positive eigenvalues.

\begin{define}
  Given a matrix $A$, we define the three quantities $n_+(A), n_-(A),
  n_0(A)$ as the number of eigenvalues of $A$ with positive, negative,
  and zero real parts, respectively.  We will refer to these as the
  indices of $A$.  (Given a vector field $\f$ with fixed point
  $\btheta^*$, we will abuse notation and refer to the indices of
  $\btheta^*$ when we mean the indices of the Jacobian of the vector
  field at $\btheta$.)
\end{define}

A straightforward calculation gives the following expression for the
Jacobian matrix
\begin{equation}\label{eq:Jacobiancoordinates}
  J_{ij}(\btheta) := \pd{f_i}{\theta_j}(\btheta) 
  = \begin{cases} \cos(\theta_i - \theta_j),&  i \neq j \\ -\sum_{k\neq i}  \cos(\theta_k-\theta_i),& i=j.\end{cases}
\end{equation}

Using the cosine angle addition formula, we have
\begin{equation}\label{eq:rank2}
  J = -D + \v \otimes \v +  \w \otimes \w
\end{equation}
where the vectors $\v$ and $\w$ are defined by
\begin{equation}\label{eq:defofvw}
  \v =  \left(\begin{array}{c} \sin(\theta_1) \\ \sin(\theta_2)
      \\ \vdots \\ \sin( \theta_n) \\ \end{array}\right),  \qquad
  \w = \left(\begin{array}{c} \cos(\theta_1) \\ \cos(\theta_2)
      \\ \vdots \\ \cos(\theta_n) \\ \end{array}\right),
\end{equation}
$D$ is the diagonal matrix 
\begin{equation}\label{eq:defofD}
  D_{ij} = \delta_{ij}\sum_k \cos(\theta_k-\theta_i),
\end{equation}
and $\otimes$ denotes the usual outer product 
$\left({\bf a} \otimes {\bf b}\right)_{ij} = a_i b_j.$ 

The Jacobian is a positive semi-definite rank two perturbation of a
diagonal matrix, a fact which figured in the analysis
of~\cite{Mirollo.Strogatz.05}. The spectrum of a low rank perturbation of an
operator can be computed explicitly in terms of the spectrum of the
unperturbed operator together with some information about the inverse
of the (unperturbed) operator, a fact usually known as the
Aronszajn-Krein formula~\cite{Simon}. Here the unpertubed operator is
diagonal and can thus be trivially inverted, and one can get a
reasonably complete description of the spectrum. 

Thus we proceed as follows.  Since the full Jacobian is a rank two
perturbation of the diagonal, the indices of $J$ and $D$ can differ by
at most two.  We will first compute the index of $D$, then compute the
difference of indices explicitly.

\begin{define}
  Given a configuration $\btheta$, we define the {\em complex order
    parameter} 
\begin{equation}\label{eq:op}
  R e^{i \psi} = \frac 1N \sum_{n=1}^N e^{i\theta_n}.
\end{equation}
\end{define}

\begin{prop}\label{prop:nD}
  Using the $O(1)$ invariance, rotate the configuration so that the
  order parameter is on the positive real axis.  Then we have
\[
n_+(-D) = \#\left\{\theta_i | \cos(\theta_i)<0\right\} 
\]
\end{prop}

\begin{proof}
Since $D$ is diagonal we just have to count the number of negative 
entries on the diagonal. We have 
\begin{equation}\label{eq:dii}
  D_{ii} = \sum_k \cos(\theta_k-\theta_i) = \cos(\theta_i) \sum_k \cos(\theta_k) + \sin(\theta_k)\sum \sin(\theta_i)
\end{equation}
which we recognize as the $(1,0)$ component of the order parameter
rotated through angle $\theta_i$. The number of times this is negative
is the number of angles in the range $(\pi/2, 3\pi/2)$.
\end{proof}

\begin{rem}
  Proposition~\ref{prop:nD} is equivalent to the calculations done in
  Section~4 of~\cite{Mirollo.Strogatz.05}.
\end{rem}

Since the perturbation is rank two, and one of the eigenvalues of the
Jacobian is zero, if $D$ is negative-definite then there is only one
eigenvalue whose sign is to be determined.


\begin{thm}\label{thm:index} 
  Suppose $D$ is invertible. Define the following quantity:
\begin{equation}\label{eq:defoftau}
  \tau = \sum_i \frac{1}{\sum_j \cos(\theta_j-\theta_i)} =  \ip\v{D^{-1}\v} +\ip\w{D^{-1}\w} =\ip{\1}{D^{-1} \1}.
\end{equation}
Then we have 
\begin{equation}\label{eq:index}
n_+(J) = n_+(-D) + \begin{cases} 0,& \tau < 2, \\ 1, & \tau > 2. \end{cases}
\end{equation}
\end{thm}

\begin{proof}
  We proceed by a homotopy argument in the spirit of the
  Birman--Schwinger Principle (see, for instance, the discussion on page 98 of Volume 4 of the text of Reed and Simon~\cite{Reed.Simon}).  Consider
  the one parameter family of operators
\begin{equation}\label{eq:defofJeta}
  J_\eta = -D + \eta(\v\otimes\v + \w\otimes\w)
\end{equation}
Clearly $J_0 =-D$ and $J_1 = J$.  Since $\v\otimes\v+\w\otimes\w$ is
non-negative, all eigenvalues are non-decreasing functions of
$\eta$. We can detect eigenvalues crossing from the left half-plane to
the right half-plane by detecting changes in the dimension of the
kernel of $J_\eta$. $J_\eta$ has non-trivial kernel if there is a
vector $\x$ with $J_\eta\x = \0$, or
\begin{equation*}
 -D \x + \eta (\ip\v\x \v + \ip\w\x \w) =0.
\end{equation*}
Solving for $\x$ gives
\begin{equation}\label{eq:x}
  \x = \eta (\ip\v\x D^{-1}\v + \ip\w\x D^{-1} \w).
\end{equation}
Taking the inner product with $\v$ and $\w$ gives
\begin{align}\label{eq:ip}
  \ip\v\x &= \eta(\ip\v\x \ip\v{D^{-1}\v} + \ip\w\x\ip\v{D^{-1}\w}),\\
  \ip\w\x &= \eta(\ip\v\x \ip\w{D^{-1}\v} + \ip\w\x\ip\w{D^{-1}\w}).
\end{align}
This can be thought of as a linear system for the quantities
\begin{equation}\label{eq:defofa1a2}
  a_1 = \ip\v\x,\q a_2 = \ip\w\x,
\end{equation}
which can be written in matrix form as
\begin{equation}\label{eq:defofmeta}
  M_\eta \l(\begin{array}{c} a_1 \\ a_2 \\\end{array}\r) := \l(\begin{array}{cc} \eta\ip\v{D^{-1}\v}-1 & \eta\ip\v{D^{-1}\w}\\ \eta\ip\w{D^{-1}\v}& \eta\ip\w{D^{-1}\w}-1\\\end{array}\r)\l(\begin{array}{c} a_1 \\ a_2 \\\end{array}\r) = \l(\begin{array}{c} 0 \\ 0 \\\end{array}\r).
\end{equation}
There exists a nontrivial solution to~\eqref{eq:defofmeta} (i.e.,
$M_\eta$ has a nontrivial kernel) if and only if $J_\eta$ itself has a
nontrivial kernel.  Thus we want to count zeroes of $\det(M_\eta)$ for
$\eta \in (0,1)$. Initially $M_0 = -{\bf I}_{2\times 2}$ has two
negative eigenvalues. At $\eta=1$ we have that $M_1$ is singular so
that one of the eigenvalues is zero. $M$ is a linear matrix value
function of $\eta$, and the linear term
\[
\l(\begin{array}{cc} \ip\v{D^{-1}\v} & \ip\v{D^{-1}\w}\\ \ip\w{D^{-1}\v}& \ip\w{D^{-1}\w}\\\end{array}\r) 
\]
is positive definite, so the eigenvalues of $M$ are increasing
functions of $\eta$. This implies that all eigenvalue crossings are
transverse left to right and so the number of positive eigenvalues of
the full matrix is determined by the non-zero eigenvalue of $M_1.$ If
this is negative there have been no eigenvalue crossings, and if this
is positive there has been one. Since one eigenvalue is zero the other
is equal to the trace, which is

\begin{equation*}
  \Tr(M_1) = -1 + \ip\v{D^{-1}\v} - 1 +  \ip\w{D^{-1}\w} = \ip{\1}{D^{-1}\1}-2 = \tau-2.
\end{equation*}

Thus if this quantity is positive a single eigenvalue has crossed into
the right half-plane and $n_+(-D)=n_+(J)+1$. If this quantity is
negative then no eigenvalues have crossed and the count is the same,
$n_+(-D)=n_+(J)$.

\end{proof}

\begin{remark}
  Since the Jacobian matrix for the Kuramoto model is of the form of a
  graph Laplacian, the Kirchoff matrix tree theorem can be applied.
  This expresses the product of the non-zero eigenvalues in terms of a
  sum over spanning trees of a product over edges in the spanning tree
  of the edge weights. In our case the edge weights may not be
  positive, but the standard proofs of the matrix tree theorem still
  hold in this case.  Our formula above suggests that the product over
  the non-zero eigenvalues of the Jacobian should be proportional to
  the quantity
\[
 \Tr(M_1) =\tau-2. 
\] 
One can check that this is the case, and actually compute the constant
of proportionality. This gives the following unusual combinatorial
identity
\begin{equation}\label{eq:frat}
\sum_{{\mathcal T}\in {\mathcal S}} \prod_{e \in {\mathcal T}} \cos(\theta_e) = \frac{2 \prod_i \sum_j \cos(\theta_i -\theta_j) - \sum_k \prod_{i\neq k} \sum_j \cos(\theta_i-\theta_j)}{\sum_{i,j} \cos(\theta_i-\theta_j)}.
\end{equation}
Here $ {\mathcal S}$ is the set of all spanning trees on $N$ points
(of which there are $N^{N-2}$), ${\mathcal T}$ is an element of $
{\mathcal S}$, a spanning tree, $e$ is an edge in the spanning tree,
and if the edge $e$ connects vertices $l$ and $m$ then
$\theta_e=\theta_l-\theta_m.$ The left-hand side of~\eqref{eq:frat}
comes from the matrix tree theorem, and the righthand side comes from
directly evaluating the product of the non-zero eigenvalues using
arguments similar to those given above. This identity is not used in
the current paper --- we really only need the formula derived in
Theorem \ref{thm:index} --- but the existence of such a formula is
extremely interesting. It suggests that there may be much more
algebraic structure in this problem than is readily apparent.
\end{remark}

Having a concise index count for the Jacobian makes it possible to
begin to make an analytical description of this region. Our first
observation is:

\begin{lem}\label{lem:sc}
  Let $\ST$ denote the region in configuration space where the there
  exists a stable synchronized solution. This region is simply
  connected and is given by the connected component of
  $\tau^{-1}([1,2))$ that contains the origin.
\end{lem}

\begin{proof}
  To see that the stability region is simply connected, and that it is
  the component containing the origin, we give an explicit
  stability-preserving retraction of any stable steady state onto the
  stable steady state $\btheta = \0$. Consider, for example, the
  retraction $\btheta(\zeta) = \zeta\btheta$.  When $\zeta=1$, we have
  the original steady state $\btheta$, and when $\zeta=0$, we have the
  stable steady state $\theta_i=0$.

  The derivative of the Jacobian with respect to the homotopy
  parameter is given by 
\[
\left(\frac{\partial J}{\partial \zeta}\right)_{ij} = \left\{\begin{array}{c} (\theta_j-\theta_i) \sin(\zeta(\theta_j-\theta_i)) \qquad i \neq j \\
-\sum_k (\theta_k-\theta_i)\sin(\zeta(\theta_k-\theta_i)). \qquad i=j 
\end{array}\right.
\]
A necessary condition for stability is that
$|\theta_i-\theta_j|\leq\pi$ for all $i,j.$ This implies that
$(\theta_j-\theta_i) \sin(\zeta(\theta_j-\theta_i))$ is a positive
quantity thus that $\frac{\partial J}{\partial \zeta}$ is positive
definite. Thus as the homotopy parameter $\zeta$ is decreased from $1$
to $0$ the eigenvalues of the Jacobian decrease, and all stable
stationary states are contractible onto the state ${\bf \theta}=0$ in
a way that increases stability.  This establishes the simple
connectedness of $\ST$ and thus, by continuity, $\SO$.

 \end{proof}

 We need something stronger than Lemma~\ref{lem:sc}. In order to
 derive lower bounds on quantities related to the probability of
 stabilization we would like a much stronger result: namely that the
 stable region is a convex set.  This is the content of the next few
 results.

\begin{define}
  Define $\kappa_i(\btheta) = \sum_j \cos(\theta_j-\theta_i)$.  Define
  $\ST$ as the subset of $\T^n$ for which we have $\kappa_i\ge 0$
  and $\tau = \sum_{i=1}^n \kappa_i^{-1} \le 2$, and $\SO$ as the image 
of $\ST$ under the map ${\bf f}.$
\end{define}

\begin{rem}
  Note that by Proposition~\ref{prop:nD} and Theorem~\ref{thm:index},
  $\SO$ is precisely the set of $\bomega$ such that $\bomega =
  \f(\btheta)$ for some $\btheta$, and the Jacobian~\eqref{eq:rank2}
  is negative semi-definite; in short, those $\bomega$ for which we
  have fully synchronous solutions.
\end{rem}

The above characterization of the stable region is difficult to use in
practice, since one needs to find the image under a reasonably
complicated map of a region defined by a transcendental equation. The
following lemma shows that the boundary of this region is actually a
connected piece of an algebraic variety.

\begin{lem}\label{lem:alg}
  The stable frequency set $\SO$ can be defined algebraically as
  follows: It is the solution set of
\begin{align}
  \kappa_i^2 + \omega_i^2 &= \sum_j \kappa_j \label{eq:ko}\\
  \kappa_i &> 0 \label{eq:pos}\\
  \sum \frac{1}{\kappa_i} &< 2 \label{eq:k2}\\
  \sum \omega_i &= 0 \label{eq:m0}
\end{align}
Since all $\kappa_i>0$,~\eqref{eq:k2} implies $\kappa_i>\frac{1}{2}$
for all $i$.  Also note that summing the first equation over $i$ shows
\[
{\boldsymbol \kappa}^2 + {\boldsymbol \omega}^2 = N \langle {\bf 1}, {\boldsymbol \kappa} \rangle
\]
or, equivalently 
\[
 |{\boldsymbol\kappa} - N {\bf 1}|^2 + {\boldsymbol \omega}^2 = \frac{N^3}{4} 
\]

and thus the entire stable region lies in a sphere of radius ${\boldsymbol \omega} \leq \frac{N^{\frac32}}{2}$.
\end{lem}

\begin{proof}
The main observation is the following: squaring $\k_i$ and $\omega_i$ 
and adding them gives 
 \begin{align*}
    \k_i^2 + \omega_i^2
    &= \sum_{j,j'} \cos(\theta_j-\theta_i)\cos(\theta_{j'}-\theta_i) + \sin(\theta_j-\theta_i)\sin(\theta_{j'}-\theta_i)  \\
    &= \sum_{j,j'} \cos(\theta_j-\theta_{j'}),
  \end{align*}
  or
  \begin{equation}\label{eq:kidentity}
    {\bf F}_i(\kappa,\omega):=\kappa_i^2 + \omega_i^2 -\sum_{j=1}^n \kappa_j=0,\mbox{ for all }i.
  \end{equation}
Computing the Jacobian of the function ${\bf F}$ we find that 
\begin{equation}\label{eq:defofM}
\frac{\partial {\bf F}}{\partial {\bf \k}} = \left(\begin{array}{ccccc} 2\k_1-1& -1 & -1 &\ldots&-1 \\ -1 & 2\k_2-1 & -1 & \ldots & -1 \\ \vdots & \vdots & \ddots & \ddots & -1\\
-1 & -1 & -1 & \ldots & 2\k_n-1  \end{array}\right)
\end{equation}

It is straightforward to check that the determinant of this Jacobian
is
\[
 \det(\frac{\partial {\bf F}}{\partial {\bf \k}}) = 2^{n-1} \prod_{i=1}^N \kappa_i \left(2 - \sum \frac{1}{\kappa_i}\right)
\]
All of the principal minors are of the same form, so it is easy to
check that this matrix is positive definite in the stable region $\sum
\frac{1}{\k_i} <2$, so an implicit function argument shows that we can
define ${\bf \k}$ as a function of $\bf \omega.$ The second equation
is equivalent to the condition $n_-(D)=0$, while the third condition
is equivalent to to the condition that $n_-(D)=n_-(-J)$. Finally the
last equation restricts to the subspace $\langle{\bf \omega},{\bf
  1}\rangle =0$.
\end{proof}

Given this algebraic representation of the stable set it is relatively 
straightforward to check that the region is convex. 

\begin{thm}\label{thm:convex}
  The set $\SO$ is convex.
\end{thm}

\begin{proof}
  We will first show that the Hessian of $\kappa_i$ (with respect
  to $\bomega$!) is negative definite for each $i$. This implies that 
$\frac{1}{\k_i}$ is a convex function. This implies that $\tau,$ 
being a sum of convex functions, is convex.  Finally, the region 
$\SO = \{\bomega\mid \tau(\bomega)\le 2\}$, is bounded by a level 
set of a convex function and is therefore a convex set. 

Differentiating~\eqref{eq:kidentity} twice gives
\begin{equation}\label{eq:kjac}
  2\k_i \pdtwo{\k_i}{\omega_k}{\omega_l} + 2\pd{\k_i}{\omega_k}\pd{\k_i}{\omega_l} + 2\delta_{k,i}\delta_{k,l} = \sum_{j=1}^N \pdtwo{\k_j}{\omega_k}{\omega_l}
\end{equation}
Recalling the definition of $\M$ in~\eqref{eq:defofM},~\eqref{eq:kjac}
can be written as a matrix equation in the form
\begin{equation}\label{eq:matrixeq}
  \M(\k)h_{kl} = 2\pd{\k_i}{\omega_k}\pd{\k_i}{\omega_l} + 2\delta_{k,i}\delta_{k,l},
\end{equation}
where we have the unknown $N$-vector
\begin{equation*}
  h_{kl} := \l\{\pdtwo{\k_i}{\omega_k}{\omega_l}\r\}_{i=1}^N.
\end{equation*}
To clear up a potential confusion: here we think of~\eqref{eq:kjac} as
an equation where we fix $k,l$ which determines a vector where we vary
the functions $\kappa_i$ in the entries of the vector.

If $\M(\k)$ is invertible, then
\begin{equation}\label{eq:hkl}
  h_{kl} = \M^{-1}\l(2\pd{\k_i}{\omega_k}\pd{\k_i}{\omega_l} + 2\delta_{k,i}\delta_{k,l}\r).
\end{equation}
We have already shown that $\M$ is invertible in the proof of
Lemma~\ref{lem:alg} above, and now we compute the inverse exactly.
Write $\M = \DD + \1\otimes \1^t$, where $\DD$ is the diagonal matrix
whose $i$th entry is $-2\k_i$.  Using the Aronszajn-Krein
formula~\cite{}, or checking directly, we have that
\begin{equation}\label{eq:AK}
  \M^{-1} = \DD^{-1} - \frac{(\DD^{-1}\1)\otimes(\DD^{-1}\1)^t}{1+\ip{\1}{\DD^{-1}\1}}.
\end{equation}
In this case, we have
\begin{equation}\label{eq:AKcoords}
  \l((\DD^{-1}\1)\otimes(\DD^{-1}\1)^t\r)_{ij} = \frac1{4\k_i\k_j},\quad 1+\ip{\1}{\DD^{-1}\1} = 1+\sum_{i=1}^n \frac1{2\k_i},
\end{equation}
so that 
\begin{equation*}
  \l(\M^{-1}\r)_{ij} = -\frac{\delta_{ij}}{2\k_i} - \frac1{1+\sum_{i=1}^n \frac1{2\k_i}}\frac{1}{4\k_i\k_j}.
\end{equation*}
In particular, every element of $\M^{-1}$ is negative for $\k\in\SO$
by~\eqref{eq:k2}.  Then we can write
\begin{equation}\label{eq:2deriv}
  \pdtwo{\k_i}{\omega_k}{\omega_l} = 2\delta_{kl}(\M^{-1})_{ik} + \sum_{j=1}^N (\M^{-1})_{ij} \l(2\pd{\k_i}{\omega_k}\pd{\k_i}{\omega_l}\r).
\end{equation}
This is a negative definite matrix: it is given by linear combination
of $N+1$ matrices with negative coefficients, and each of these
matrices are manifestly positive definite --- the first is diagonal
with positive entries, and the rest are of the form $(\nabla
\k_j)\otimes(\nabla \k_j)$.

Finally, we note that if $\k_i$ has a negative definite Hessian, then
$\k_i^{-1}$ has a positive definite Hessian.  To see this, write
$\nu_i = 1/\k_i$, and we have
\begin{equation}\label{eq:nu}
  \pdtwo{\nu_i}{\theta_\a}{\theta_\b}(\btheta) = \frac{-\k_i(\btheta)\pdtwo{\k_i}{\theta_\a}{\theta_b}(\btheta) + 2 \pd{\k_i}{\theta_\a}(\btheta)\pd{\k_i}{\theta_\b}(\btheta)}{\k_i(\btheta)^3}.
\end{equation}
Since $\k_i >0$, this implies that the Hessian of $\nu_i$ is positive
definite.  Since $\tau = \sum \k_i^{-1}$ is convex the stable set
$\ST$, which is bounded by a level set of $\tau$, is convex.
\end{proof}

\subsection{Lower Bounds on the stable set}

We next characterize a subregion of $\ST$ for which we are guaranteed
stability. The first region takes the form of a curvilinear polytope,
and is the image of a cube under the map $\bf f$. We will show that
the vertices of this polytope actually lie on the boundary of $\ST$ as
well. This region is more difficult to deal with analytically, so we
introduce a second (flat) polytope which is determined by the convex
hull of the vertices of the curvilinear polytope. This polytope is the
Voronoi cell of the root lattice $A_n$ and as such its properties are
well-studied~\cite{}.  This allows us to get very good estimates
geometric quantities such as the volume, the probability that a random
vector lies in the stable synchronous region, etc.

\begin{define}
  We define the stable cube $\mathcal{C}$ to be  
  \begin{equation}\label{eq:defofC}
    {\mathcal C} : = \l\{\btheta\mid  \btheta \in [0,\pi/2]^N\r\}.
  \end{equation}
  and the set $\mathcal{V}$ to be the vertex set of the the cube with
  two vertices removed:
  \begin{equation}\label{eq:defofV}
    \mathcal{V} := \{\btheta| \forall i, \theta_i\in \{0,\pi/2\}\}/ \{(0,0,0,\ldots,0), (\frac{\pi}{2},\frac{\pi}{2},\frac{\pi}{2},\ldots,\frac{\pi}{2})\}.
  \end{equation}
\end{define}

\begin{rem}
  One can consider vertices of the $n$-cube as representing the
  partition of the $n$ oscillators into two sets. In this
  interpretation $\mathcal{V}$ represents the partition of the $n$
  oscillators into two non-empty sets. We will see below that the set
  $\mathcal{V}$ represents points in ${\mathbb T}^N$ that map to
  points on the boundary of $\SO$. The points $(0,0,0,\ldots,0)$ and
  $(\frac{\pi}{2},\frac{\pi}{2},\frac{\pi}{2},\ldots,\frac{\pi}{2})$
  map to the interior of $\SO$ (in fact to the origin - the most
  stable configuration).
\end{rem}

\begin{lem}\label{lem:CV}
  The Jacobian is positive semi-definite on the cube $\mathcal{C}$.
  Moreover,
  \begin{equation}\label{eq:kerJ}
    \dim(\ker(J)) = \begin{cases} 1, &   \btheta \in {\mathcal C}/ {\mathcal V},  \\
      2, &  \btheta \in  {\mathcal V}.\end{cases}
  \end{equation}
\end{lem}

\begin{proof}
  If $\theta_i \in (0,\pi/2)$, then $\th_i-\th_j\in(-\pi/2,\pi/2)$ for
  all $i,j$ and thus $\cos(\th_i-\th_j) > 0$ for all $i,j$.  Then the
  Jacobian matrix~\eqref{eq:defofJ} takes the form of a weighted graph
  Laplacian on the complete graph with positive weights.  (Said
  another way, $J$ has zero row sums and positive off-diagonal
  entries.)  The dimension of the kernel of a graph Laplacian is equal
  to the number of connected components of the graph~\cite{}. Since
  none of the weights vanish, the dimension of the kernel is equal to
  $1$.

  When $\btheta$ is on the boundary of $\mathcal{C}$, some of the
  weights may vanish, if the corresponding angles differ by exactly
  $\pi/2$. At the level of the graph, this amounts to decomposing the
  vertex set of the graph into two sets (one corresponding to
  $\theta_i=0$ and the other to $\theta_i=\frac{\pi}{2}$) and breaking
  all connections between the two sets. For the complete graph on $N$
  nodes, this will disconnect the graph only if { all} points
  belong to one of these sets and neither one is empty.  These
  configurations correspond exactly to points in $\mathcal{V}$.
\end{proof}

We have defined a cube in configuration space where the stationary
solutions are at least marginally stable, and the vertices of this
region actually lie on the stability boundary. Next we characterize
the shape of the corresponding region in frequency space.

\begin{prop}\label{lem:voronoi}
  The image of $\mathcal{V}$ under the map $\f$ gives $2^N-2$ distinct
  frequency vectors.  These vectors form (up to scaling) the vertices
  of $V(A_n)$, the Voronoi cell of the root lattice $A_N$.
  Equivalently $V(A_N)$ is the projection of the cube
  $[-\frac{\gamma N}{2},\frac{\gamma N}{2}]^N$ onto the $(N-1)$-dimensional plane
  normal to $\1$ in $R^N$.
\end{prop}

\begin{proof}
  It is easy to check directly that the image of $\mathcal{V}$ under
  the map ${\bf f}$ consists of vectors of the following form: if the
  vertex has $i$ angles of $0$ and $j$ angles of $\frac{\pi}{2}$ then
  the image is a permutation of the vector
\[
{\bf \omega} = \gamma (\underbrace{i,i,i,\ldots,i}_{j~~{\rm times}},\underbrace{-j,-j,-j,\ldots,-j}_{i~~{\rm times}})
\]
where $i+j=N$. It is a reasonably well-known fact (see, for instance,
Conway and Sloane~\cite{CS} Chapter 21.3B) that the Voronoi cell
around the origin of the root lattice $A_N$ has vertices given by
vectors of the form
\[
{\bf v} = (\underbrace{\frac{i}{N},\frac{i}{N},\frac{i}{N},\ldots,\frac{i}{N}}_{j~~{\rm times}},\underbrace{-\frac{j}{N},-\frac{j}{N},-\frac{j}{N},\ldots,-\frac{j}{N}}_{i~~{\rm times}})
\] 
and is the projection of the unit cube $[-\frac{1}{2},\frac{1}{2}]^N$ onto the $(N-1)$-plane of mean zero vectors. 

\end{proof}

\begin{rem}

  The root lattice $A_N$ and the associated polytope arise in a
  surprising number of areas of mathematics including Lie algebras and
  root systems, Coxeter groups, coding theory, etc. The appearance
  here is perhaps not so surprising since the symmetry group of the
  Kuramoto problem, $S_N \times S_2$ (corresponding to permutation of
  the oscillators and ${\bf \theta} \mapsto -{\bf \theta}$), is the
  same as the symmetry group of the $A_N$ lattice.

\end{rem}

\begin{corr}
  $\mathrm{cl}(\SO)$ contains $V(A_n)$ and the image of the vertex set
  $\f(\mathcal{V})$ lies on the boundary of $\SO$.
\end{corr}
 \begin{proof}
This is clear from the proposition above: a convex polytope is equal to the convex hull of its vertices, 
and the convexity of $\SO$ implies that the convex hull of a collection of boundary points is 
contained in the closure.   
\end{proof}

It is worth noting that there is a relatively easily computable set of 
points lying on the boundary on the stability region. While we have not had 
occasion to use this fact in the present paper this fact may prove useful 
for improving the current bounds, and is presented here.

\begin{prop}
  The stable region contains the (suitably scaled) dual polytope to
  the Voronoi cell $V(A_N)$. The vertices of this dual polytope are
  given by the vectors
  \[
  \pm \omega_N(1,0,0,\ldots,-1)
  \]
  and all permutations, where $\omega_N$ is the number 
  \[
  \omega_N := \frac{1}{16\sqrt{2}} \left(\sqrt{32 + (N-1)^2} + 3 (N-1)\right)\sqrt{16 + (N-1) \sqrt{32 + (N-1)^2} - (N-1)^2}
\]
\end{prop}

\begin{proof}
The configurations which correspond to these frequencies are those with 
a group of $N-2$ angles at the origin and two others offset 
symmetrically from these:
\begin{equation*}
\theta_i =
\begin{cases} -x, & i=0,\\ 0, & i \in (2,N-1),\\ x,& i=N\end{cases}
\end{equation*}

The corresponding frequency vector is 
\[
\omega_i = ((N-1) \sin(x) + \sin(2x),0,0,\ldots,0,-(N-1) \sin(x) - \sin(2x)).
\] 
If one maximizes the length of this frequency vector over all $x$ then one finds the 
expression above.

\end{proof}

\section{Upper and Lower Bounds on $\gamma$}\label{sec:upperlower}

We first consider a formal argument as to the sizes of the vertex set
in frequency space, i.e. the set $\f(\mathcal{V})$.

\begin{defn}
  We define the frequency vectors $\omin, \omax$ as follows:
\begin{equation}\label{eq:defofomin}
  \omin := \pm(1,1,\dots,1,-(N-1))^t
\end{equation}
for all $N$.  We define $\omax$ differently if $N$ is even or odd.  If
$N$ is even, we define $\omax$ as that vector whose first $N/2$
components are $N/2$, and whose last $N/2$ components are $-N/2$.  If
$N$ is odd, we define $\omax$ to be the vector with the first
$(N-1)/2$ entries are $(N+1)/2$, and whose remaining entries are
$-(N-1)/2$.  In summary,
\begin{equation}\label{eq:defofomax}
  \omax:= \begin{cases} (\underbrace{N/2,N/2,\dots,N/2}_{N/2},\underbrace{-N/2,-N/2,\dots,-N/2}_{N/2}),& N\mbox{ even,}\\ \underbrace{(N+1)/2,\dots,(N+1)/2}_{(N-1)/2},\underbrace{-(N-1)/2,\dots,-(N-1)/2}_{(N+1)/2},& N\mbox{ odd.}\end{cases}
\end{equation}

\end{defn}

\begin{lem}\label{lem:vertices}
We have 
\begin{equation}\label{eq:vertexminmax}
  \min_{\btheta\in \mathcal{V}} \norm{\f(\btheta)}^2 = N(N-1),\quad
  \max_{\btheta\in \mathcal{V}} \norm{\f(\btheta)}^2 = \begin{cases} N^3/4,& N\mbox{ even,}\\N(N^2-1)/4, & N\mbox{ odd.}\end{cases}
\end{equation}
Moreover, the $\bomega$ which minimize (resp.~maximize)
in~\eqref{eq:vertexminmax} are permutations of $\omin$
(resp.~$\omax$).
\end{lem}

\begin{proof}
  First note that $\av{\mathcal{V}}=2^N-2$ since we choose all
  vertices of the cube except two.  Given $\btheta\in\mathcal{V}$, let
  $i\in 1,\dots,N-1$ be the number of entries of $\btheta$ which are
  equal to $0$ and $j=N-i$ be the number which are equal to
  ${\pi}/{2}$.  Everything is the same up to permutation, so replace
  $\btheta$ with the vector
  \begin{equation*}
    \btheta^* = (\underbrace{0,0,\dots,0}_i,\underbrace{\pi/2,\pi/2,\dots,\pi/2}_j)^t.
  \end{equation*}
  Then
  \begin{equation}\label{eq:defofbomega}
    \bomega^* = \f(\btheta^*) = (\underbrace{j,j,\dots,j}_{i},\underbrace{-i,-i,\dots,-i}_{j})^t 
  \end{equation}
  We have
  \begin{equation*}
    \norm{\bomega^*}^2 = i^2 j + i j^2 = i j N.
  \end{equation*}

  Clearly, to minimize this, we choose $i=N-1$ and $j=1$ (or
  vice-versa), and this gives $\bomega^* = \omin$, and $\norm{\omin}^2
  = N(N-1)$.  

  To maximize this, it depends on the parity of $N$; if $N$ is even,
  we choose $i=j=N/2$, and if $N$ is odd, we choose
  $i=\lfloor{N/2}\rfloor=(N-1)/2$.  In either case we obtain $\omax$
  as defined above.  We calculate that for $N$ even, $\norm{\omax}^2 =
  N^3/4$, and for $N$ odd, $\norm{\omax}^2 =N(N^2-1)/4$.

\end{proof}

\begin{rem}
  Notice that, for all $N$, $\norm{\omax}^2 = N^3/4 + o(N^3)$.  The
  $\btheta$'s corresponding to $\omin, \omax$ have a relatively simple
  characterization.  The minimum vertex can be described as ``a herd
  of sheep and a lone wolf'': $(N-1)$ oscillators at $\theta_i=0$ and
  $1$ oscillator at $\theta_i=\frac{\pi}{2}.$ The maximum vertex can
  be described as the state of ``two competing cliques''.
%
  Finally, note that in any case, the vertices are really only
  distinguished by the size of the partitions $i$ and $j$, and thus
  there are $\lfloor\frac{N+1}{2}\rfloor$ different types of vertex.
\end{rem}

\begin{lem}\label{lem:1/g}
  We have
\begin{equation}\label{eq:1/g}
  \gmin(N) = \frac1{\sup_{\btheta\in\ST}\norm{\f(\btheta)}},\q   \gmax(N) = \frac1{\inf_{\btheta\in\ST}\norm{\f(\btheta)}},\q 
\end{equation}
\end{lem}

\begin{proof}
  Recall the definition of $\gmin(N),\gmax(N)$:
\begin{equation}\tag{\ref{eq:defofgmingmax}}
    \gmin(N) := \inf_{\substack{\bomega\in\R^N\\\av{\bomega} = 1}} \gstar(\bomega), \q \gmax(N) := \sup_{\substack{\bomega\in\R^N\\\av{\bomega} = 1}} \gstar(\bomega)
\end{equation}
Moreover, notice that multiplying the right-hand side of~\ref{eq:K} by
a scalar does not change anything about the existence or stability of
a fixed point.  Choose $\norm{\bomega}=1$, and if we try to solve
\begin{equation*}
  \g \f(\btheta) = -\bomega,
\end{equation*}
we have to have $\g\norm{\f(\btheta)} = 1$, or 
\begin{equation*}
  \g \ge \frac1{\sup_{\btheta\in\ST}\norm{\f(\btheta)}}.
\end{equation*}
Since $\ST$ is open, this bound is saturated and gives us $\gmin(N)$.
The opposite argument works for $\gmax(N)$.
\end{proof}

\begin{thm}\label{thm:minmax}
  Recall the definition of $\gmin,\gmax$ in~\eqref{eq:defofgmingmax}
  above.  Then
\begin{equation}\label{eq:gmin}
  \gmin(N) = \begin{cases}2N^{-3/2},&N\mbox{ even,}\\ 2N^{-3/2} + O(N^{-5/2}),&N\mbox{ odd,}\end{cases}
\end{equation}
and
\begin{equation}\label{eq:gmax}
  \frac1{\sqrt{N(N-1)}} \le \gmax(N) \le \frac{\sqrt{2}}{\sqrt{N(N-1)}}
\end{equation}
\end{thm}

\begin{proof}
  We start with $\gmin(N)$.  Recall the identity~\eqref{eq:kidentity};
  summing both sides of this equation over $i$ gives us
\begin{equation}\label{eq:kidentitysummed}
  \sum_{i=1}^N \omega_i^2 = N\sum_{i=1}^N \k_i - \sum_{i=1}^N \k_i^2 = \sum_{i=1}^N (N\k_i-\k_i^2).
\end{equation}
The right-hand side of~\eqref{eq:kidentitysummed} is maximized when we
choose $\k_i = N/2$ and therefore we have the global bound
$\norm\bomega^2 \le N^3/4$.  From Lemma~\ref{lem:1/g}, we can deduce
that $\gmin(N) \ge 2/N^{-3/2}$.

If $N$ is even, then this bound is obtained at $\omax$, and therefore
\begin{equation}\label{eq:supbomega}
  \sup_{\bomega\in\SO} \norm\bomega^2 = \frac{N^3}4,
\end{equation}
and thus $\gmin(N) = 2N^{-3/2}$.

If $N$ is odd, then $\omax$ does not attain the $N^3/4$ bound.
However, because of $\omax$, we know that
\begin{equation*}
  \sup_{\btheta\in\ST} \norm{\f(\btheta)}^2 \ge \frac{N^3-N}4,
\end{equation*}
and therefore 
\begin{equation*}
  \gmin(N) \le \frac4{N^3-N} = \frac4{N^3}\frac{1}{1-1/N^2} = \frac4{N^3}(1+O(N^{-2})).
\end{equation*}

In a similar vein, we have that 
\begin{equation*}
  \inf_{\btheta\in \ST} \norm{\f(\btheta)} \le \min_{\btheta\in\mathcal{V}}\norm{\f(\btheta)} = \sqrt{N(N-1)},
\end{equation*}
and so $\gmax(N) \ge 1/(\sqrt{N(N-1)})$. The upper bound comes from the fact that the stability region 
$\SO$ contains the Voronoi polytope. It is straightforward to calculate the radius of the sphere inscribed in the Voronoi polytope: it is $\frac{1}{\sqrt{2}}$ independent of dimension $n$. This gives a lower 
bound for $\inf_{\theta} {\bf f}({\bf \omega})$, and thus an upper bound on $\gmax$ of 
\[
\gmax(N) \leq \frac{\sqrt{2}}{N(N-1)}
\]

\end{proof}

\begin{rem}
  While the above estimates give the correct order of magnitude, and
  in certain cases the exact value, of $\gmin$ and $\gmax$ we believe
  that the following are the exact values.
\begin{align*}
 \gmax(N) &= \frac{1}{\sqrt{N(N-1)}} \\
 \gmin(N) &=  \begin{cases}  \dfrac2{N^{\frac32}}, & N~~{\rm even}, \\
\dfrac{8 \sqrt{2(N-1)} }{\left(\sqrt{8 N^2-16 N+9}+3\right) 
\sqrt{4 N^2-8 N + 3 +\sqrt{8 N^2-16N+9}}},&  N~~{\rm odd}.\\ \end{cases}
\end{align*}

We conjecture that the configuration with $N-1$ oscillators having
angle $\theta_i=0$ and one having angle $\theta_i=\frac{\pi}{2}$ is a
global minimizer of the frequency $\omega$ over the marginally stable
set. It is easy to check that it is a local minimum, and numerical
results indicate that (at least for small N) it is a global
minimum. For $\gmin$ the even case is tight, from the calculation
above. The conjectured value for $N$ odd requires some comment. For
the case $N$ odd there is a configuration which is {\em not} a vertex
which has a larger value of $\omega$ than any vertex. The
configuration which gives this is as follows: a single oscillator at
$\theta=0$, a two groups of $\frac{N-1}{2}$ placed symmetrically on
either side at angle $x$. This gives the magnitude of the frequency as
the solution to the following maximization problem
\[
\omega^* = \sup_{x} \sqrt{N-1} (\sin(x) + \frac{N-1}{2}\sin(2x))
\] 
whose solution is 
\[
\omega^*= \frac{\left(\sqrt{8 N^2-16 N+9}+3\right) 
\sqrt{4 N^2-8 N + 3 +\sqrt{8 N^2-16N+9}}}{8 \sqrt{2(N-1)} }.
\]
It is straightforward to check that for large $N$ this is
asymptotically $\frac{N^{\frac32}}{2}+O(N^{\frac12})$, and would give
$\gmin(N) = 2N^{-3/2} + O(N^{-5/2})$, consistent with
Theorem~\ref{thm:minmax}.
\end{rem}

\section{Large $N$ limit}\label{sec:Gaussian}

In many of the problems of interest one is interested in the case
where the number of oscillators is large, and the frequencies are
chosen from a specific probability distribution.  In this section we
establish rigorously that the interesting scaling for the Kuramoto
problem when the frequencies are chosen independently is not the
classical scaling $N^{-1}$, but actually the scaling $\varphi(N) :=
\frac{\sqrt{2\ln(N)}}{N+1}$.  More specifically, we prove the
following:

\begin{thm}\label{thm:limit} Suppose that the frequencies $\omega_i$ are independent
  identically distributed (i.i.d.) Gaussian random variables with unit
  variance, i.e. we assume that the $\omega_i$ are chosen
  independently, and that
  \begin{equation}\label{eq:Gaussian}
    \P(\omega_i \in (a,b)) = \frac1{\sqrt{2\pi}}\int_a^b e^{-x^2/2}\,dx.
  \end{equation}

  Let $\psync(\gamma,N)$ denote the probability that the
  system~\eqref{eq:K} has a stable state with {\em all} oscillators
  locked. Then we have the following dichotomy:
  \begin{itemize} 
  \item If $\delta < 1$  then $\lim_{N\rightarrow \infty} \psync(\delta\varphi(N),N)=0$. 
  \item  If $\delta \geq 2$  then $\lim_{N\rightarrow \infty} \psync(\delta\varphi(N),N)=1$.
  \end{itemize}
\end{thm}

The basic strategy is straightforward: given our previous results on
the shape of the stability domain we establish upper and lower bounds
on the probability that a Gaussian random frequency will lie in the
stable region.  We prove each of these statements separately in two
lemmas.

\begin{lem}\label{lem:ub}
  Suppose that the components of the frequency are i.i.d Gaussian as
  in~\eqref{eq:Gaussian}. The probability that the system will exhibit
  stable synchronization satisfies the upper bound
  \begin{equation}\label{eq:ub}
    \psync(\gamma,N) \leq \sqrt{N} \left(\erf(\gamma N/\sqrt{2})\r)^{N-1}.
  \end{equation}
\end{lem}

\begin{proof}
  The vector $\bomega$ is distributed according to the multivariate
  Gaussian measure
  \begin{equation}\label{eq:multiG}
    \P(\bomega\in A) = (2 \pi)^{-\frac{N+1}{2}}\int_A \exp(-\av{\x}^2/2)\, d\x.
  \end{equation}

  Since we have moved to the co-rotating frame and we have assumed
  that $\sum \omega_i =0$, we make the following (non-orthogonal!)
  change of variables.  First choose $\chi_i$ as in~\eqref{eq:multiG}
  above, then write
  %
  \begin{equation}\label{eq:defofq1}
    \omega_i = \chi_i - v,\q v:= \frac1N \sum_{i=1}^N \chi_i, \q i=1,\dots,N.
  \end{equation}
  Note then that 
  \begin{equation}\label{eq:defofq2}
    \omega_N = -\sum_{i=1}^{N-1} \omega_i,\mbox{ or, }\chi_N = v-\sum_{i=1}^{N-1} \omega_i.
  \end{equation}
  Note that we've written the $\chi$'s in terms of $\omega_i$ with
  $i=1,\dots,N-1$ and $v$; these will be the new variables.  As we
  prove in Lemma~\ref{lem:Jacobian} below, the Jacobian of this change
  of variables is $N$.  Moreover, the quadratic form transforms as
\begin{equation}\label{eq:qf}
\begin{split}
    \sum_{i=1}^{N} \chi_i^2 
    &= \sum_{i=1}^{N-1} (\omega_i+v)^2 + \chi_N^2 \\
    &= \sum_{i=1}^{N-1}\omega_i^2 + 2v\sum_{i=1}^{N-1}\omega_i + (N-1)v^2 + \l(v-\sum_{i=1}^{N-1} \omega_i\r)^2\\
    &= \sum_{i=1}^{N-1}\omega_i^2 + Nv^2 + \l(\sum_{i=1}^{N-1} \omega_i\r)^2.
\end{split}
\end{equation}
Therefore, for any set $A$, if we denote the transformation
in~(\ref{eq:defofq1},~\ref{eq:defofq2}) as $q$, then we have
  \begin{equation}\label{eq:xform}
  \int_{q(A)} e^{-\av\bchi^2/2}\,d\bchi = \int_A e^{-\frac12 \l[ \sum_{i=1}^{N-1}\omega_i^2 + Nv^2 + \l(\sum_{i=1}^{N-1} \omega_i\r)^2\r]}\,N\,d\bomega,
\end{equation}
where the $N$ comes from the Jacobian of the transformation.  To
determine the domain of integration $A$, note that if we have a fixed
point,
\begin{equation}\label{eq:fplt}
  \av{\omega_i} =\g \av{\sum_{j=1}^N \sin(\theta_j-\theta_i)} \le \g N,
\end{equation}
so we take our domain to be
\begin{equation}\label{eq:defofA}
  A:= \av{\omega_i} < \g N\mbox{ for }i=1,\dots, N-1,\q v\in \R,
\end{equation}
and it is a necessary condition for synchronization that
$\boldsymbol\chi\in A$.  Thus we compute an upper bound on the
probability for synchronization:
  \begin{align*}
    \psync(\gamma,N) &\leq (2 \pi)^{-N/2}\int_{[-\gamma N,\gamma
      N]^{N-1}\times {\mathbb R}} e^{-\frac{1}{2}\left[
        \sum_{i=1}^{N-1} \omega_i^2 + \left(\sum_{i=1}^{N-1}
          \omega_i\right)^2 +
        Nv^2\right] }\,N\prod_{i=1}^N d\omega_i\,dv\\
    &\leq (2 \pi)^{-N/2}N\int_{[-\gamma N,\gamma N]^{N-1}\times
      {\mathbb R}} e^{-\frac{1}{2}\left[ \sum_{i=1}^{N-1} \omega_i^2 +
        Nv^2\right] }\,\prod_{i=1}^N d\omega_i\,dv\\
    &= \sqrt{N} \prod_{i=1}^{N-1} \l(\frac1{\sqrt{2\pi}}\int_{-\gamma N}^{\gamma N} e^{-\omega_i^2/2}\,d\omega_i\r)\times \l(\frac{\sqrt{N}}{\sqrt{2\pi}}\int_{-\infty}^{\infty} e^{-Nv^2}\,dv\r)\\
    &=\sqrt{N} \prod_{i=1}^{N-1} \erf(\gamma N/\sqrt{2}) \times 1 =
    \sqrt{N}\l(\erf(\gamma N/\sqrt{2}) \r)^{N-1}.
  \end{align*}

\end{proof}

\begin{lem}\label{lem:Jacobian}
  If we define the transformation as
  in~(\ref{eq:defofq1},~\ref{eq:defofq2}), then its Jacobian is
\begin{equation}\label{eq:Jacobian}
  \av{\frac{\partial (\chi_1,\ldots, \chi_{N}) }{\partial      (\omega_1,\ldots, \omega_{N-1}, v)}} = N.
\end{equation}
\end{lem}

\begin{proof}
  We compute
\begin{align*}
  \pd{\chi_i}{\omega_j} &= \delta_{ij},\q \pd{\chi_i}{v} = 1,\q i,j=1,\dots,N-1,\\
  \pd{\chi_N}{\omega_i} &= -1,\q \pd{\chi_N}{v} = 1.
\end{align*}
Thus this Jacobian (which we will call $G_N$) is the constant matrix
which we write  in block-diagonal form:
\begin{equation}\label{eq:defofGN}
  G_N=\l(\begin{array}{cc} I_{N-1}& 1_{(N-1)\times 1}\\-1_{1\times (N-1)} &1 \\\end{array}\r),
\end{equation}
where $I_k$ is the $k\times k$ identity matrix, $1_{a\times b}$ is the
$a\times b$ matrix of all ones, etc.  By elementary row operations
(adding the sum of the first $N-1$ rows to the last) the Jacobian can
be reduced to
\[
\l(\begin{array}{cc} I_{N-1}& 1_{(N-1)\times 1}\\0_{1\times (N-1)} &N \\\end{array}\r).
\]
This matrix is upper-triangular and we can read off the determinant as
$N$.
\end{proof}

The next proposition gives a lower bound for the probability by identifying a 
region where a sufficient condition for stability holds, over which the 
Gaussian integral can be evaluated rather explicitly. The proof of this is 
straightforward but requires some facts about polytopes in ${\mathbb R}^N$
\begin{lem}\label{lem:lb}
The probability of synchronization satisfies the following lower bound:
\begin{equation}\label{eq:lb}
  \psync(\gamma,N) \geq \l(\erf\l(\frac{\gamma N}{2\sqrt{2}}\r)\r)^{N}.
\end{equation}

\end{lem}

\begin{proof}
  The proof here uses the fact that the polytope $V(A_N)$, which we
  know to be contained in the (closure of) the stable region, is the
  projection onto the mean zero hyper-plane of the cube in one higher
  dimension. This, together with the orthogonal invariance of the
  Gaussian, give the result.

  We use the notation of the previous lemma, that ${\boldsymbol
    \chi}\in {\mathbb R}^N$ a frequency vector which is {\em not}
  assumed to have zero mean, and ${\boldsymbol \omega} = {\boldsymbol
    \chi} - \frac{\langle{\bf 1}, \boldsymbol\chi \rangle}{N}$ is the
  orthogonal projection onto the mean zero subspace. By the orthogonal
  invariance of the Gaussian we have that
 \[\int_{V(A_N) \times {\mathbb R}} (2
\pi)^{-\frac{N}{2}} e^{-\frac{\Vert\boldsymbol\chi \Vert^2 }{2} }\, d\boldsymbol\chi = \int_{V(A_N)} (2
\pi)^{-\frac{N-1}{2}} e^{-\frac{\Vert\boldsymbol\omega \Vert^2 }{2} }\, d\boldsymbol\omega.
\]
Since the stability region contains the polytope $V(A_n)$, we have the
inequality
\[
 \psync(\gamma,N) \geq \int_{V(A_N) \times {\mathbb R}} (2
\pi)^{-\frac{N}{2}} e^{-\frac{\Vert\boldsymbol\chi \Vert^2 }{2} }\, d\boldsymbol\chi = \int_{V(A_N)} (2
\pi)^{-\frac{N-1}{2}} e^{-\frac{\Vert\boldsymbol\omega \Vert^2 }{2} }\, d\boldsymbol\omega.
\]
Next note that since the polytope $V(A_N)$ is the projection of the
cube $\l[-\frac{N\gamma}{2},\frac{N\gamma}{2}\r]^N$ onto the $(N-1)$-
dimensional plane normal to ${\bf 1}$, we necessarily have that the cube
is contained in $V(A_N)\times{\mathbb R}$, the cylinder in ${\mathbb
  R}^N$ with cross-section $V(A_N)$:
\[
\l[-\frac{\gamma}{2},\frac{\gamma}{2}\r]^N \subset V(A_N)\times{\mathbb R}.
\]
This in turn shows that 
\[
 \psync(\gamma,N) \geq \int_{V(A_N) \times {\mathbb R}} (2
\pi)^{-\frac{N}{2}} e^{-\frac{\Vert\boldsymbol\chi \Vert^2 }{2} }\, d\boldsymbol\chi \geq \int_{[-\frac{N\gamma}{2},\frac{N\gamma}{2}]^N} (2
\pi)^{-\frac{N}{2}} e^{-\frac{\Vert\boldsymbol\chi \Vert^2 }{2} }\, d\boldsymbol\chi = (\erf(\frac{N \gamma}{2\sqrt{2}})^N .
\]

\end{proof}

{\bf Proof of Theorem~\ref{thm:limit}.}  Assume that $\g =
\delta\varphi(N)$, with $\delta<1$.  Using Lemma~\ref{lem:ub}, we have
that
\begin{equation}\tag{\ref{eq:ub}}
  \psync(\gamma,N) \leq  \sqrt{N}\left(\erf(\gamma N/\sqrt{2})\r)^{N-1}.
\end{equation}
The standard asymptotic expansion for the error function
gives~\cite{}, for large $x$:
\begin{equation}\label{eq:erf}
  \erf(x) = 1 - \frac{e^{-x^2}}{x\sqrt{\pi}}(1+O(x^{-2})),
\end{equation}
so we have
\begin{equation}\label{eq:erfN}
  \erf(\gamma N/ \sqrt{2}) = \erf(\delta\sqrt{\ln(N)}) =  1 - \frac{\exp(-\d^2\ln(N))}{\d\sqrt{\pi\ln(N)}}(1+O(\ln(N)^{-1})).
\end{equation}
This means that there exist $C_1\in(0,1), \d<\tilde\d<1$ such that,
for $N$ sufficiently large,
\begin{equation*}
  \erf(\gamma N/\sqrt{2}) = 1 - C_2\exp(-{\tilde\d}^2\ln(N)) \le 1-C_2 N^{-\tilde{\d}^2}
\end{equation*}
A straightforward application of L'H\^opital's Rule shows that, if
$\tilde\d<1$, then
\begin{equation*}
  \l(1-C_2 N^{-\tilde{\d}^2}\r)^N \asymp \exp(-C_3 N^{1-\tilde\d^2}),
\end{equation*}
and therefore, if $\d<1$, in the limit as $N\to\infty$, the right-hand
side of~\eqref{eq:ub} goes to zero as $N\to\infty$. (The $\sqrt N$ in
front and one power of $\erf$ in the back will not change anything for
large $N$.)

Now assume $\g = \delta\varphi(N)$ with $\d>2$, and the argument is
similar.  Using Lemma~\ref{lem:lb}, we have
\begin{equation}\tag{\ref{eq:lb}}
  \psync(\gamma, N) \ge \erf\l(\frac{\gamma N}{2\sqrt2}\r)^N.
\end{equation}
Again using~\eqref{eq:erf}, we plug $\gamma = \delta\varphi(N)$
into a single term and obtain
\begin{equation*}
  \erf(\gamma N/(2\sqrt2)) = \erf(\delta\sqrt{\ln(N)}/2) = 1-\frac{\exp(-\frac{\d^2}4\ln(N))}{\frac\d2\sqrt{\pi\log(N)}}\l(1+O(\ln(N)^{-1})\r).
\end{equation*}
This means that there exists $C>0, 2<\d<\tilde{\d}$ such that for
$N$ sufficiently large,
\begin{equation*}
  \erf(\gamma N/(2\sqrt2)) \ge 1-C N^{-\tilde{\d}^2/4}.
\end{equation*}
Again using l'H\^{o}pital's Rule, we have that if $\tilde{\d}>2$, then
this goes to 1 as $N\to\infty$.

\qed

Finally, we prove that, in the standard scaling, the probability of
choosing a fully synchronous solution is zero and, moreover, that it
goes to zero exponentially fast.  Specifically, we prove:
\begin{prop}\label{prop:exp}
  For all $\d >0$, there exist $C_1, C_2 \in \R$ with
\begin{equation}\label{eq:exp}
  \psync(\delta/ N, N) \le C_1 e^{-C_2(\delta) N}.
\end{equation}
\end{prop}

\begin{proof}
  Using Lemma~\ref{lem:ub}, we have that 
\begin{equation*}
  \psync(\d/N,N) \le \sqrt{N}\erf(\d/\sqrt{2})^{N-1}
\end{equation*}
Choose $\erf(\d/\sqrt2) < \alpha < 1$, and write $\beta =
\alpha/\erf(\d/\sqrt2)$.  Then a calculus argument shows that if we
define $C_1 > (2 e \log(1/\beta))^{-1/2}$, then $C_1 > \sqrt{N}
\beta^{N-1}$ for all $N$, or
\begin{equation*}
  \sqrt{N}\erf(\d/\sqrt{2})^N < C_1 \alpha^{N-1}.
\end{equation*}
Choose $C_2 = -\log\alpha$ and pull a power of $\alpha$ into $C_1$,
and we are done.

\end{proof}

\begin{figure}[ht]
\begin{centering}
\includegraphics[width=6in]{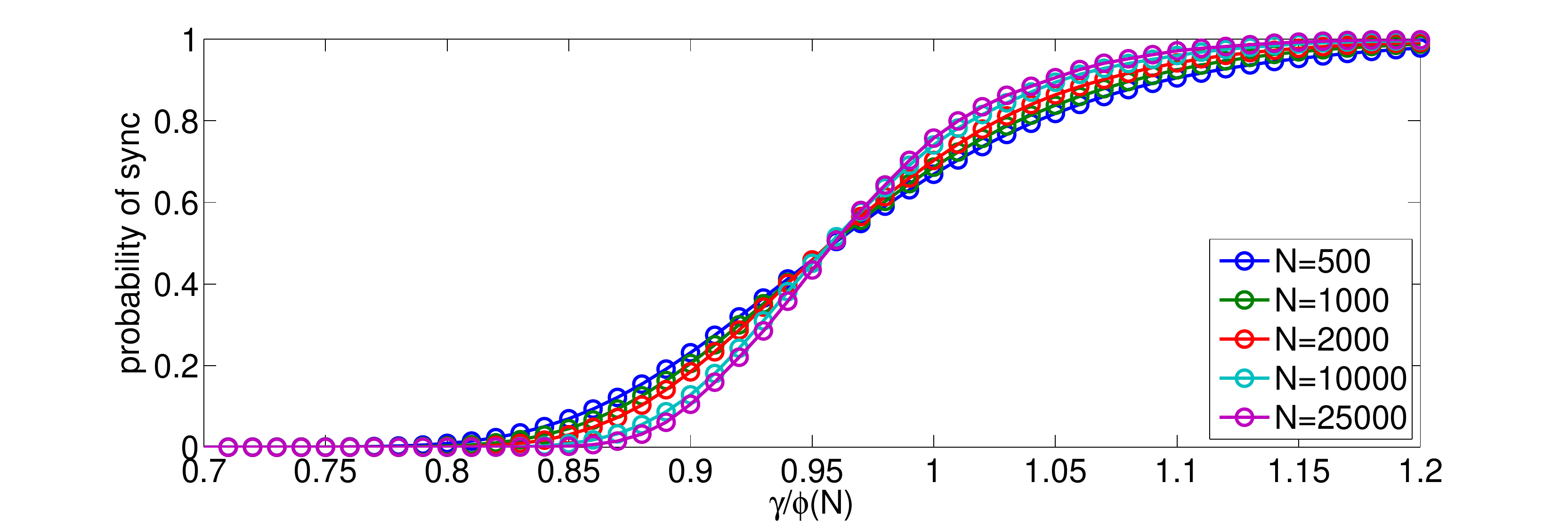}
  \caption{A numerical simulation of the probability of full synchrony
    in the Kuramoto model in the newly identified scaling.  }
\label{fig:probdense}
\end{centering}
\end{figure}

Finally we close this section with some numerical simulations.  The
figure \ref{fig:probdense} depicts a Monte-Carlo simulation of the
Kuramoto problem with $N=1000$, $N=2000$, $N=10,000$ and $N=25,000$
oscillators.  The figures were generated as follows: for each
realization a direction was generated uniformly on the
$N-1$-dimensional sphere, and the distance to the boundary of the
stability region was computed. In this situation the formulation of
Mirollo and Strogatz was found to be more computationally efficient,
and this was solved numerically using a bisection method. For each
such direction, given the distance to the stability boundary the
conditional probability that a Gaussian random vector in that
direction would lie within the stability region could be calculated,
and averaging over all realizations gives a numerical approximation to
the probability of synchronization. For each of the graphs $5000$
realizations were used.

The numerics agree well with the analytical results. We have shown
analytically that in the limit of large $N$ the probability of full
synchrony is $0$ for $\gamma<1$. The numerics suggest that
\[
 \lim_{N \rightarrow \infty} P(\gamma \phi(N),N) = \left\{\begin{array}{c} 0 \qquad \gamma < \gamma^* \\ 1 \qquad \gamma>\gamma^* \end{array}\right.
\]
with a critical coupling constant of roughly $\gamma^* \approx 1$.
This statement is consistent with, but sharper than, the conclusion of
Theorem~\ref{thm:limit}.

\section{Examples}\label{sec:examples}

\subsection{Comprehensive example for three oscillators}

For $N=3$ there are $2^3-2=6$ vertices. The corresponding frequency
vectors are given by
\begin{align*}
& \vec \omega_1 = (1,1,-2)^t \\
& \vec \omega_2 = (1,-2,1)^t \\
& \vec \omega_3 = (-2,1,1)^t \\
& \vec \omega_4 = (-1,-1,2)^t \\
& \vec \omega_5 = (-1,2,-1)^t \\
& \vec \omega_6 = (2,-1,-1)^t. \\
\end{align*}
The plane orthogonal to the vector $(1,1,1)$ is spanned by the vectors

\begin{equation*}
  {\bf e}_1=(1,0,-1)/\sqrt2,\quad {\bf e}_2 = (1,-2,1)/\sqrt{6}. 
\end{equation*}
In this basis the frequency vectors have the representations $\pm
\sqrt{6} {\bf e}_2, \sqrt6(\pm \frac{\sqrt3}{2} {\bf e}_1 \pm
\frac{1}{2} {\bf e}_2 )$, which are obviously the vertices of a
regular hexagon of side length $\sqrt6.$ As shown in the lemma, these
represent local minima of distance on the surface of marginal
stability.

The phase diagram for three oscillators is summarized in
Figure~\ref{fig:fly}.  Since the map from the configuration space
${\mathbb T}^{n-1}$ to the frequency space ${\mathbb R}^{n-1}$ has
degree zero it follows that the number of preimages of a given point
$\omega$ (the number of configurations with a given frequency) is
even, half of which have positive (reduced) Jacobian determinant and
half of which have negative Jacobian determinant.  The latter are, of
course, always unstable. The former have even index but may not be
stable. Outside the shaded region there are no fully synchronized
solutions. As one crosses the stability boundary a pair of
synchronized solutions are created: one is stable (index zero) and one
is unstable (index one). The guaranteed stable region, the image of
the cube $[0,\frac{\pi}{2}]^3$ under the frequency map, is also
plotted: it is a dark curve just inside the stable region.

As shown earlier we can derive expressions for the frequency vectors
which are hardest and easiest to synchronize. The hardest vectors to
synchronize have $(n-1)$ oscillators traveling together with no phase
shift and $1$ oscillator leading (or trailing) by phase
$\frac{\pi}{2}.$ These solutions have frequency vectors
\[
\boldsymbol \omega = (\pm1,\pm1,\mp2)
\] 
and permutations. These points are marked by the six dots on the
boundary.  These points lie on a circle of radius $\sqrt6$. The
solutions which are easiest to synchronize have one oscillator at
phase $0$, one leading by $x$ and one trailing by $x$ with
corresponding frequency vector
\[
\boldsymbol \omega = \left(\pm(\sin(x) + \sin(2x)), 0, \mp(\sin(x) +
  \sin(2x))\right)
\]
plus permutations. The maximum length frequency vector of this form is attained at 
\[
x = \arccos(\frac{1}{8}(\sqrt{33}-1))
\]  
which gives a frequency vector of length 
\[
|\omax(3)| = \frac{\sqrt{2}}{16}(3 +
\sqrt{33})\sqrt{\frac{1}{2}\left(15 + \sqrt{33}\right)} \approx 2.49
\]
It is interesting that, in the case of three oscillators the stability
region is very close to a circle and the phase transition to the
fully-synchronized state is quite sharp: there are {\em no} fully
synchronized solutions for $\omega > \frac{\sqrt{2}}{16}(3 +
\sqrt{33})\sqrt{\frac{1}{2}\left(15 + \sqrt{33}\right)} \approx 2.49$
and {\em all} solutions synchronize for $\omega < \sqrt{6} \approx
2.45$, a relative frequency shift of about $1.6\%$.

On the interior of the stability region there are two more secondary 
bifurcation curves, which look like a pair of curvilinear triangles 
rotated through angle $\frac{2 \pi}{6}$ relative to one another. 
As one crosses these curves a pair of unstable 
solutions are created, one of index $1$ and one of index $2$. 
Thus the frequency plane can be subdivided as follows:
\begin{itemize}
\item Region 1: Two solutions, one of index $0$ and one of index $1$.
\item Region 2: Four solutions, one of index $0$, two of index $1$, 
one of index $2$.
\item Region 3: Six solutions, one of index $0$, three of index $1$, 
two of index $2$.  
\end{itemize}

\begin{figure}[ht]
\begin{centering}
\includegraphics[width=4in]{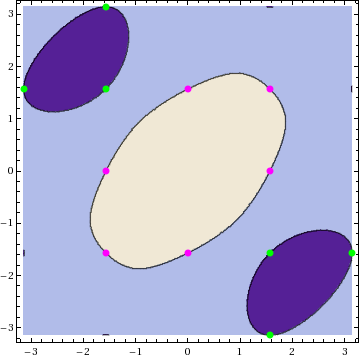}
\caption{A plot of the stability regions for the three particle
  system.  In the light beige central region the Jacobian is of index
  0. In the surrounding pale blue region the Jacobian is of index
  1. In the two dark blue islands the Jacobian is of index 2. The
  marked points depict special points. The six central magenta ones
  depict the last frequencies to be stabilized, while the six green
  ones represent points on the boundary between the regions of index 2
  and index three.}
\label{fig:dope}
\end{centering}
\end{figure}

Also shown is a related figure in the the configuration space. Here
the configuration space is colored according to the stability of the
given configuration (note that by the $O(1)$ symmetry the first
oscillator can be chosen at $\theta_1=0$). The light-colored central
region represents the stable solutions. It is, as was shown,
convex. The surrounding gray region represents the index $1$
solutions, i.e.~those with a single unstable direction. Both of these
regions cover the range of the frequency map. Finally the two darkest
regions represent the two triangular regions where there exist
solutions of index two.

\begin{figure}[ht]
\begin{centering}
\includegraphics[width=4in]{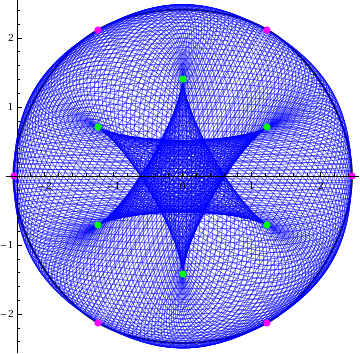}
\caption{The frequency map. The magenta points correspond to the
  analogous points in the previous plot: the boundary of the region
  between the region of index zero and the region of index one
  corresponds to the boundary of the range. These points denote the
  last frequency vectors to exhibit synchrony. The green points denote
  similarly distinguished points on the boundary between the regions
  of index 1 and index 2. Also shown is the guaranteed stable region,
  which is tangent to the boundary at the magenta.}
  \label{fig:fly}
\end{centering}
\end{figure}

\subsection{Example for four oscillators}
The case of four oscillators is the first with different types of vertices. 
The vertices are of two types: the first are vertices of the form 
\[
\vec \omega = (\pm 1,\pm1,\pm1,\mp3)
\]
and permutations, which have length $\sqrt{12}$. These are local minima of the 
length of the frequency vector constrained to the surface of marginal 
stability, and represent the hardest frequencies to stabilize. There are 
eight such vertices. The second type of vertex has frequency vectors of the 
form 
\[
\vec \omega  = (\pm 2,\pm 2, \mp2,\mp2)
\]
and permutations, which have length $4$. There are six such vertices
and these represent the easiest vectors to stabilize.  These points
together form the vertices of a polytope known as the rhombic
dodecahedron. It is the Voronoi cell for the $A_3$ lattice - the face
centered cubic lattice.

The figure depicts the region of guaranteed stability - the image of the 
cube $[0,\frac{\pi}{2}]$ under the frequency projection map. Since this 
is a symmetric nonlinear projection of the cube in ${\mathbb R}^4$ into 
${\mathbb R}^3$ it is perhaps not surprising that it takes the form of 
a curvilinear rhombic dodecahedron, since the corresponding linear projection 
gives the (flat) rhombic dodecahedron. The vertices typified by the 
frequency $\omega = (1,1,1,-3)$ are those where three of the curvilinear 
rhombic faces come together, while those vertices typified by frequencies 
like $(2,2,-2,-2)$ are those where four rhombic faces meet.  The actual 
stability region (not depicted) is somewhat larger, and resembles an octahedron which 
is tangent to the region of guaranteed stability at the fourteen vertices listed above.

\begin{figure}[ht]
\begin{centering}
\includegraphics[width=4in]{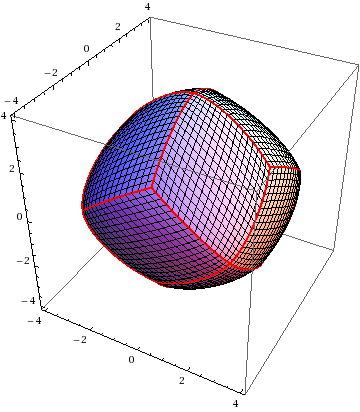}
\caption{ A region of guaranteed stability for the 
Kuramoto problem with four oscillators. The region is a rhombic dodecahedron 
with curvilinear faces. The vertices are all 
points of marginal stability, where there is a zero eigenvalue of multiplicity 
two or more.  
}
  \label{fig:phat}
\end{centering}
\end{figure} 

\section{Scaling and Extreme Value Statistics}\label{sec:scaling}

In the preceeding section we showed that the critical scaling for 
full synchronization differs from the classical scaling by a logarithmic 
factor: that one should consider 
\[
\frac{d{\bf \theta}}{dt} = {\bf \omega} + \frac{\gamma\sqrt{\log(N^2)}}{N}
\]
rather than the more usual $\frac{1}{N}$ scaling. In this section 
we give a short heuristic argument as to why this is the correct 
scaling, which is connected with the extreme value statistics of the 
frequencies\cite{Resnick, BC}.

Let us recall the basics of extreme value statistics. Give a collection of 
idenpendent and identically distributed random variables $\omega_i$ 
the extreme values statistics concerns the distribution of the 
quantity $M_N=\max(\omega_1,\omega_2,\omega_3,\ldots \omega_N).$ 
The Fisher-Tippett-Gnedenko theorem\cite{Resnick} characterizes the rescaled distribution
of such a quantity. It says that if the distribution of a suitably 
rescaled $M_N$  converges to a non-degenerate distribution $G(z)$:
\[
{\mathbb P}(M_N-b_N)/a_N\leq z) {\rightarrow} G(z) \mbox{ as }{N \rightarrow \infty}
\]
then $G(z)$ is one of the following three distributions, the Gumbell,
Frechet and Weibull distributions.
\begin{align}
  G_1(z) &= \exp(\exp(-z)), \\
  G_2(z) &= \begin{cases} 0, & z\leq 0, \\ \exp(-z^{-\alpha}), & z>0,\\ \end{cases}\\
  G_3(z) &= \begin{cases} \exp(-z^{\alpha}), & z<0, \\ 1, & z>0. \end{cases}
\end{align}
We claim that the preceeding calculation shows that the probability of
full synchrony for the Kuramoto models is bounded above and below by
extreme value statistics for the Gaussian, and that the slightly
anomolous scaling seen in this problem is a reflection of the extreme
value statistics.

It is easy to see that for the Kuramoto problem one has the following
obvious estimate for synchronous solutions. From the formula for the
frequencies in the classical scaling one has
\[
\omega_i - \omega_j = \frac{1}{N}\sum_k \sin(\theta_k-\theta_i) - \sin(\theta_k-\theta_j) 
\]
leading to the easy estimate 
\[
\max_i(\omega_i) - \min_i(\omega_i) \leq 2 
\]
that must hold in order for there to be a fully synchronous solution.
This gives an upper bound for the probability in terms of the
distribution of
\[
\max_i(\omega_i) - \min_i(\omega_i),  
\]
a kind of mean adjusted extreme value statistic. This holds for any choice 
of distribution of the frequencies. For the Gaussian case it is clear that 
the correct scaling is for this statistic is 
\begin{equation}\label{eq:mm}
  \max_i(\omega_i) - \min_i(\omega_i) \propto \sqrt{\log{N}}.
\end{equation}
This gives a motivation for the modified scaling necessary for full
synchrony.  For Gaussian distibutions of frequencies the lower bound
is also of the same form. The lower bound we proved in the preceeding
section estimates the probability of synchrony in terms of the
probability that the frequency lies in a Voronoi cell of the $A_N$
lattice.  

So then might pose the question of where the standard scaling might be
applicable?  We note that the anomalous scaling occurs because of the
fact that while the typical sample of a Gaussian is $O(1)$, the
maximum of a sample of size $N$ (the {\em first order statistic}) is
significantly larger than this, being $O(\sqrt{\log N})$.  So one
could modify the original question: instead of requiring full
synchrony for all oscillators, which leads to a condition
like~\eqref{eq:mm}, we could ask the question of whether or not we
have ``all but one'' oscillator synchronize, or even ``all but $k$''
oscillators synchronize.  This would then be governed by the typical
size of the second order statistic or the $k+1$st order statistic.
However, for fixed $k$ and $N\to\infty$, these all have the same
scaling, so we conjecture that for these problems we would also obtain
the scaling seen here.

However, if we modified even further and asked for solutions such that
some fixed fraction of oscillators synchronized (e.g. for large $N$,
we require that $0.9N$ oscillators synchronized), then this would come
from the typical size of the 90th percentile of the sample, and this
is $O(1)$.  Thus we further conjecture that this type of question
would lead to the classical scaling.  We will consider these questions
in future work.

\section*{Acknowledgments}

The authors would like to thank Yulij Baryshnikov for useful discussions.  RELD
was partially supported by NSF grant CMG-0934491. JCB and MJP were partially 
supported by NSF grant DMS-0807584.

\bibliographystyle{plain} \bibliography{Kuramoto}

\end{document}